\documentclass[twocolumn]{autart}  

\usepackage{amsmath,amssymb}       % Math symbols and environments
\usepackage{bm}                    % Bold math symbols (e.g., \bm{x})
\usepackage{enumitem}
\usepackage{algorithm}
\usepackage{algpseudocode}

\usepackage{etoolbox}
\makeatletter
\expandafter\patchcmd\csname tabular*\endcsname{\hfil}{\hfil\extracolsep{\fill}}{}{}
\makeatother

\usepackage{graphicx}              % Include figures
\usepackage{epstopdf}              % EPS to PDF conversion (if using .eps)
\usepackage{svg}
\usepackage{subfig}

\newtheorem{definition}{Definition}

\newtheorem{theorem}{Theorem}

\newtheorem{proposition}{Proposition}
\newtheorem{corollary}{Corollary}
\newtheorem{remark}{Remark}
\newtheorem{assumption}{Assumption}

\usepackage{color}                 % Color support (for text, not figures)
\usepackage{url}                   % For URLs in references

\begin{document}

\begin{frontmatter}

\title{Density-Driven Optimal Control: Convergence Guarantees for Stochastic LTI Multi-Agent Systems}
\thanks[footnoteinfo]{This work was supported by NSF CAREER Grant CMMI-DCSD-2145810.}

\author[Paestum]{Kooktae Lee}\ead{kooktae.lee@nmt.edu}    % Add the 

\address[Paestum]{Department of Mechanical Engineering, New Mexico Institute of Mining and Technology, Socorro, NM 87801, USA }  % Please supply                                              

\begin{keyword}                           % Five to ten keywords,  
Multi-Agent Systems;
Stochastic Control;
Optimal Transport;
Wasserstein Distance;
Area Coverage.
% chosen from the IFAC 
\end{keyword}                             % keyword list or with the 
                                          % help of the Automatica 
                                          % keyword wizard

\begin{abstract}
This paper addresses the decentralized non-uniform area coverage problem for multi-agent systems, a critical task in missions with high spatial priority and resource constraints. While existing density-based methods often rely on computationally heavy Eulerian PDE solvers or heuristic planning, we propose Stochastic Density-Driven Optimal Control (D$^2$OC). This is a rigorous Lagrangian framework that bridges the gap between individual agent dynamics and collective distribution matching. By formulating a stochastic MPC-like problem that minimizes the Wasserstein distance as a running cost, our approach ensures that the time-averaged empirical distribution converges to a non-parametric target density under stochastic LTI dynamics. A key contribution is the formal convergence guarantee established via reachability analysis, providing a bounded tracking error even in the presence of process and measurement noise. Numerical results verify that Stochastic D$^2$OC achieves robust, decentralized coverage while outperforming previous heuristic methods in optimality and consistency.
\end{abstract}

\vspace{-.2in}

\end{frontmatter}

%%%%%%%%%%%%%%%%%%%%%%%%%%%%%%%%%%%%%%%%%%%%%%%%%%%%%%%%%%%
\section{Introduction}
\vspace{-0.116in}

The multi-agent area coverage problem has recently garnered significant attention due to its broad range of applications, including search and rescue, environmental monitoring, infrastructure inspection, smart farming, and planetary exploration. The central challenge lies in guiding a team of agents to maximize coverage performance within a given domain. While traditional strategies like the lawnmower path have been widely used for uniform coverage, these methods are often inefficient in large-scale environments and may become infeasible when resources such as operation time, agent number, and communication range are constrained. Consequently, non-uniform area coverage has emerged as a more practical alternative, especially under limited resources and mission-specific priorities.

{
Recent approaches to non-uniform area coverage aim to match agent behavior with a reference distribution encoding spatial priorities. Spectral Multiscale Coverage (SMC) \cite{mathew2009spectral,mathew2011metrics} leverages Fourier-based metrics to align long-term visitation frequencies based on ergodic control. However, ergodicity is only theoretically achieved as $t \to \infty$, making it ill-suited for missions with strict finite-time constraints. To address this, density-based optimal control frameworks have been explored. For instance, \cite{chen2016optimal} and \cite{terpin2024dynamic} study optimal transport over dynamical systems using Eulerian-based gradient flows or dynamic programming in probability spaces. While theoretically rigorous, these methods often rely on centralized computations of a global density flow and are primarily restricted to linear dynamics or simplified state spaces, which are computationally prohibitive for high-dimensional decentralized swarms. More recently, \cite{rapakoulias2025steering} employs mean-field Schr\"{o}dinger bridges with Gaussian Mixture Models (GMM) to steer agent populations. However, such approaches are limited by the parametric assumptions of GMMs and the high complexity of solving coupled PDEs, making them less flexible for non-parametric target densities.

Density-Driven Control (D$^2$C) has emerged as a practical Lagrangian-based alternative for matching empirical agent distributions with a target density \cite{kabir2021efficient,lee2022density}. By adopting a Lagrangian perspective, where the swarm is represented as a collection of discrete, identifiable particles rather than a continuous field, D$^2$C effectively circumvents the curse of dimensionality and the heavy computational burden associated with solving high-dimensional PDEs in Eulerian density control. Despite its computational efficiency and scalability, the reliance on heuristic planning often limits its performance and optimality, particularly in complex environments with high uncertainty. Specifically, existing D$^2$C frameworks frequently overlook the dynamic coupling between individual agent trajectories and the collective distribution matching objective under stochastic perturbations. This gap necessitates a unified decentralized framework capable of simultaneously accounting for strict physical operational constraints and the inherent stochasticity arising from process and measurement noises.

To bridge this gap, we propose the Stochastic D$^2$OC, a rigorous optimal control framework for non-uniform coverage in stochastic linear time-invariant (LTI) multi-agent systems. Unlike previous heuristic methods, our framework reformulates distribution matching within the Optimal Transport (OT) context by minimizing the Wasserstein distance as a running cost in a stochastic MPC-like formulation. This leads to an optimal control law that ensures the time-averaged spatial distribution converges to a reference density in a provably consistent manner, even under severe noise. Furthermore, by characterizing the reachable set of the agent dynamics, we provide a formal convergence analysis ensuring the empirical distribution remains within a bounded neighborhood of the target. Numerical simulations under aggressive noise and perturbations demonstrate that Stochastic D$^2$OC achieves robust, decentralized matching while strictly satisfying individual dynamics and physical constraints.
}

%%%%%%%%%%%%%%%%%%%%%%%%%%%%%%%%%%%%%%%%%%%%%%%%%%%%%%%%%%%

\section{Preliminaries}
\vspace{-0.116in}

{
\noindent \textbf{Notation:} Let \(\mathbb{R}\) and \(\mathbb{Z}\) be the sets of real and integer numbers, with \(\mathbb{Z}_{>0}\) and \(\mathbb{Z}_{\geq 0}\) denoting positive and non-negative integers, respectively. 
For a matrix \(A \in \mathbb{R}^{m \times n}\), its transpose is \(A^\top\). 
The Euclidean norm is \(\|\mathbf{x}\|\), and the trace is \(\operatorname{tr}(A)\). 
The identity and zero matrices of size \(n\) are \(\mathbf{I}_n\) and \(\mathbf{0}_n\). 
A Gaussian distribution with mean \(\mu\) and covariance \(\Sigma\) is \(\mathcal{N}(\mu, \Sigma)\), and \(\|U\|_R := \sqrt{U^\top R U}\) denotes the weighted norm for \(R \succ 0\). 
The operators \(\mathrm{diag}(\cdot)\) and \(\mathrm{blkdiag}(A_h)_{h=r}^{r+H-1}\) construct diagonal and block-diagonal matrices, respectively. 
Finally, \(\otimes\) and \(\odot\) represent the Kronecker and Hadamard products.
}

Consider a multi-agent system, where each agent is governed by discrete-time stochastic LTI dynamics, evolving over a discrete-time index \(k \in \mathbb{Z}_{\ge 0}\), as follows:
\begin{equation}
\begin{aligned}
x_i^{k+1} &= A_i x_i^k + B_i u_i^k + w_i^k, \quad w_i^k \sim \mathcal{N}(0,\Sigma_{i,w}), \\\label{eq:dyn}
y_i^k &= C_i x_i^k + v_i^k, \quad v_i^k \sim \mathcal{N}(0,\Sigma_{i,v}),
\end{aligned}
\end{equation}
where \(x_i^k \in \mathbb{R}^n\) is the state, \(u_i^k \in \mathbb{R}^m\) is the control input, and \(y_i^k \in \mathbb{R}^d\) is the output at time \(k\). The matrices \(A_i \in \mathbb{R}^{n \times n}\), \(B_i \in \mathbb{R}^{n \times m}\), and \(C_i \in \mathbb{R}^{d \times n}\) define the system dynamics, control influence, and observation model, respectively. The process noise \(w_i^k\) and measurement noise \(v_i^k\) are independent, zero-mean Gaussian random variables with covariances \(\Sigma_{i,w}\) and \(\Sigma_{i,v}\), capturing model uncertainty and sensing errors.

To ensure well-posedness of the control problems considered in this paper, we impose the following standard assumption on each agent's dynamics.

\medskip
\begin{assumption}\label{assump:controllability}
    For each agent \(i\), the pair \((A_i, B_i)\) is completely controllable.
\end{assumption}

\begin{assumption}\label{assump:stability}
    For each agent \(i\), \(A_i\) is at least marginally stable, i.e., all eigenvalues lie in the closed unit disk, and any eigenvalue on the unit circle is non-defective (i.e., the algebraic and geometric multiplicities are equal).
\end{assumption}

% \begin{assumption}\label{assump:observability}
%     For each agent $i$, the pair $(A_i, C_i)$ is completely observable.
% \end{assumption}

These assumptions ensure that each agent is fully controllable and that its open-loop dynamics are at least marginally stable, which, together with bounded noise covariance, forms the basis for the convergence analysis.

\subsection{Wasserstein Distance and the Kantorovich Optimal Transport Problem}
\vspace{-0.116in}
To develop the D$^2$OC strategy, we utilize optimal transport theory \cite{villani2008optimal}. The $p$-Wasserstein distance between two discrete measures $\rho$ and $\nu$ on a metric space $(\mathcal{X}, d)$ is defined via the Kantorovich problem:
\begin{equation}
\mathcal{W}_p(\rho, \nu) = \bigg( \min_{\pi_{ij} \geq 0} \sum_{i=1}^M \sum_{j=1}^N \pi_{ij} \, d(y_i, q_j)^p \bigg)^{\frac{1}{2}},
\label{eq:discrete_wasserstein}
\end{equation}
subject to $\sum_{j=1}^N \pi_{ij} = \alpha_i$, $\sum_{i=1}^M \pi_{ij} = \beta_j$, and $\sum_{i,j} \pi_{ij} = 1$, where $\pi_{ij}$ is the transport plan from agent points $y_i \in \mathbb{R}^d$ to fixed sample points $q_j \in \mathbb{R}^d$ with masses $\alpha_i, \beta_j$. We set $p=2$ and use Euclidean distance for $d(\cdot, \cdot)$.

To facilitate D$^2$OC, we distinguish between evolving agent trajectories $y_i^k$ and fixed reference locations $q_j$. Each of the $n_a$ agents generates $M_i$ points over a finite horizon, totaling $M = \sum M_i$ points. Each agent $i$ maintains its own capacity $\alpha_i^k = 1/M_i$ and locally estimates sample capacities $\beta_{i,j}^k$ (initially $1/N$). These capacities $\beta_{i,j}^k$ are updated locally to reflect coverage history, i.e., $\beta_{i,j}^k$ decreases as agent $i$ visits or passes near $q_j$, effectively tracking explored regions.

To evaluate global density-driven coverage, we define the time-averaged empirical distribution $\rho^k$ and the reference distribution $\nu$ as:
\begin{equation}
\rho^k := \frac{1}{(k+1)n_a} \sum_{t=0}^{k} \sum_{i=1}^{n_a} \delta_{y_i^{t}}, \quad \nu := \frac{1}{N} \sum_{j=1}^N \delta_{q_j},
\label{eq:agent_distribution}
\end{equation}
where $\delta_y$ is the Dirac measure at $y$. The objective of D$^2$OC is to coordinate agents to minimize the global Wasserstein distance $\mathcal{W}_2(\rho^k, \nu)$ within a finite operational time, achieving effective and robust decentralized non-uniform coverage.

\subsection{Three-stage D$^2$OC Overview}
\vspace{-0.116in}
Although the core methodology presented here differs fundamentally from prior works \cite{kabir2021efficient,kabir2021wildlife,lee2022density}, which lack guarantees on optimality or convergence, the overall structure for realizing D\textsuperscript{2}OC follows a similar high-level logic, summarized in the following three stages:

\begin{enumerate}[leftmargin=*, itemindent=4pt]
    \item \textbf{Local Target Sample Selection and Optimal Control}:  
    At each time step \(k\), each agent selects some sample points among the reference samples $\{q_j\}$ based on proximity and remaining weight (favoring less-visited points). 
    Each agent then computes a control input using the selected local target samples, to minimize the local Wasserstein distance, subject to dynamic and input constraints.

    \item \textbf{Weight Update}:  
    Next time step after moving, the agent solves a local Wasserstein distance to update sample weights, reducing those recently covered and encouraging coverage of under-explored areas.

    \item \textbf{Weight Sharing}:  
    Agents exchange sample weights within communication range, synchronizing coverage estimates via a min-weight consensus that selects the smallest weight for each sample point. This ensures agreement on maximal coverage progress, enhancing coordination and reducing redundant visits.
\end{enumerate}

The first two stages are executed independently by each agent without communication, while the final stage enables collaborative coverage via information exchange among neighbors. This iterative cycle guides agents to match the reference density over time within dynamic and communication constraints. Although this paper primarily focuses on the first stage and theoretical convergence analysis of D$^2$OC, further details about three-stage approach can be found in \cite{kabir2021efficient,kabir2021wildlife,lee2022density}.

%%%%%%%%%%%%%%%%%%%%%%%%%%%%%%%%%%%%%%%%%%%%%%%%%%%%%%%%%%%

\section{Optimal Control Minimizing Local Wasserstein Distance}
\vspace{-0.116in}

This section presents the optimal control law for the first stage of the three-stage D$^2$OC framework. Let \(\mathcal{S}_i^k \subseteq \{1, \dots, N\}\) be the index set of sample points assigned to agent \(i\), and \(\mathcal{Q}_i^k = \{q_j \in \mathbb{R}^d \mid j \in \mathcal{S}_i^k\}\) the corresponding local targets with capacity weights \(\beta_{i,j}^k \geq 0\). The selection of local targets is detailed later based on reachability analysis.

\medskip
\begin{definition}\label{def:output relative degree}
(Output Relative Degree in Stochastic Discrete-Time LTI Systems)
Consider the stochastic discrete-time LTI system \eqref{eq:dyn}. The \emph{output relative degree} \(r \in \mathbb{Z}_{>0}\) is the smallest positive integer such that
$C_i A_i^{r-1} B_i \neq 0$, and $C_i A_i^{\ell-1} B_i = 0, \, \forall \ell = 1, \ldots, r-1.
$
\end{definition}

This means the control input affects the output starting at step \(k+r\). Thus, the expected squared \textit{local Wasserstein distance} for agent $i$ is constructed from time \(k+r\) with the prediction horizon \(H \in \mathbb{Z}_{>0}\) as follows:
\vspace{-.3in}

{\small
\begin{align*}
\mathbb{E}\left[\sum_{h=r}^{H+r-1}(\mathcal{W}_i^{k+h})^2\right] := \sum_{h=r}^{H+r-1}\sum_{j \in \mathcal{S}_i^{k+h}} \pi_j^{k+h} \, \mathbb{E}\left[\|y_i^{k+h} - q_j\|^2\right],
%\label{eq:local_wasserstein_cost}
\end{align*}
}
\vspace{-.2in}

where \(\pi_j^{k+h}\) is the transport weight from agent \(i\) to local sample \(q_j\) given by the local optimal transport plan at time $k+h$. This captures stochastic system dynamics and expected transport cost from predicted agent positions to assigned targets.
Leveraging this property, we establish the following result.

{
\noindent \textbf{Relationship to \cite{lee2025connectivity}:} The following quadratic reformulations and optimal control laws (see Proposition~\ref{proposition:equiv} and Theorem~\ref{theorem:W_min_opt_con}) maintain notation consistency with \cite{lee2025connectivity} while providing a fundamental stochastic generalization. Unlike the deterministic study in \cite{lee2025connectivity}, explicitly incorporating noise statistics via the expectation operator is a critical prerequisite for the rigorous convergence analysis in Section~\ref{sec:convergence}.
}

\medskip
\begin{proposition}\label{proposition:equiv}
Let $\mathcal{S}_i^{k+h}$ denote the index set for the local sample points for agent $i$ at time $k$ and the optimal transport plan  \(\pi_j^{k+h} \geq 0\) be given for all \(j \in \mathcal{S}_i^{k+h}\) and for \(h = r, \ldots, H+r-1\). Define the weighted barycenter at time \(k+h\) as
$
\bar{q}_i^{k+h} := \frac{1}{\sum_{j \in \mathcal{S}_i^{k+h}} \pi_j^{k+h}}\sum_{j \in \mathcal{S}_i^{k+h}} \pi_j^{k+h} q_j.
$

Then, the following equality holds: 
\vspace{-.2in}

{\small
\begin{align}
&\mathbb{E} \bigg[ \sum_{h=r}^{H+r-1} \left(\mathcal{W}_i^{k+h}\right)^2 \bigg] 
= \mathbb{E} \left[ \left\| \boldsymbol{\Omega}_i^{k|r:H} (Y_i^{k|r:H} - \bar{Q}_i^{k|r:H}) \right\|^2 \right]\nonumber\\
& \quad+ \mathrm{const.}, \nonumber\\
&\text{where  }\,
Y_i^{k|r:H} := \begin{bmatrix}
(y_i^{k+r})^\top & \cdots & (y_i^{k+H+r-1})^\top
\end{bmatrix}^\top \in\mathbb{R}^{dH}, \nonumber\\
&\bar{Q}_i^{k|r:H} := \begin{bmatrix}
(\bar{q}_i^{k+r})^\top & \cdots & (\bar{q}_i^{k+H+r-1})^\top
\end{bmatrix}^\top\in\mathbb{R}^{dH},\nonumber\\
&\boldsymbol{\Omega}_i^{k|r:H} := 
\mathrm{blkdiag}\!\big(
\sqrt{\textstyle\sum_{j\in\mathcal{S}_i(k+h)}\!\pi_j(k+h)}\,\mathbf I_d
\big)_{h=r}^{r+H-1}.
\label{eq:Y,Q,Omega}
\end{align}
}
\end{proposition}

\begin{proof}
By expanding the quadratic term,
\vspace{-.25in}

{\small
\begin{align*}
&\mathbb{E} \Bigg[ \sum_{h=r}^{H+r-1} \left(\mathcal{W}_i^{k+h}\right)^2 \Bigg] = \mathbb{E} \Bigg[ \sum_{h=r}^{H+r-1} \sum_{j \in \mathcal{S}_i^{k+h}} \pi_j^{k+h}
\| y_i^{k+h} - q_j \|^2 \Bigg]\\
% \qquad\qquad\qquad
&= \mathbb{E} \left[ \sum_{h=r}^{H+r-1}
\left( \sum_j \pi_j^{k+h} \right) \| y_i^{k+h} - \bar{q}_i^{k+h} \|^2 + C_h \right]\\
&= \mathbb{E} \left[ \sum_{h=r}^{H+r-1}
\left( \sum_j \pi_j^{k+h} \right) \| y_i^{k+h} - \bar{q}_i^{k+h} \|^2 \right] + \text{const.}
\end{align*}
}
\vspace{-.2in}

where each \(C_h := \sum_j \pi_j^{k+h} \| q_j - \bar{q}_i^{k+h} \|^2\) is constant w.r.t. \(U_i^{k|H}\). Stacking the terms into vectors and using the diagonal matrix \(\boldsymbol{\Omega}_i^{k|r:H}\) yields the compact quadratic form.
\end{proof}

{
\begin{remark}[Local Transport Plan] 
The transport plan $\pi_j^{k+h}$ is an internal variable computed locally by agent $i$ to determine the optimal mass distribution to target samples $\{q_j\}$. Given a constant mass $\alpha_i = 1/M_i$ allocated at each step, the agent independently solves a matching problem to minimize the local Wasserstein distance. The existence and uniqueness of this optimal $\pi_j^{k+h}$ are guaranteed by the optimal matching strategy established in \cite[Prop. 1]{kabir2021wildlife}.
\end{remark}
}

{
The control objective is to minimize the expected squared Wasserstein cost along with a weighted input penalty:
\vspace{-.5in}

\begin{equation}
\min_{U_i^{k|H}} J(U_i^{k|H}) := \mathbb{E}\bigg[\sum_{h=r}^{H+r-1} (\mathcal{W}_i^{k+h})^2 \bigg] + \|U_i^{k|H}\|^2_R \label{eq:d2oc_optimization_prob}
\end{equation}
where $U_i^{k|H} \in \mathbb{R}^{mH}$ is the stacked input and $R \succ 0$ is a positive definite weighting matrix. To avoid the intractability of global Wasserstein minimization, we focus on the local distance $\mathcal{W}_i^{k+h}$. Based on Proposition~\ref{proposition:equiv}, the cost in \eqref{eq:d2oc_optimization_prob} is reformulated into a strictly convex quadratic form using the following matrix definitions:
\vspace{-.3in}

{\small
\begin{align}
\Theta_i &:= \begin{bmatrix}
C_i A_i^{r-1} B_i & \mathbf{0}_{dm} & \cdots & \mathbf{0}_{dm} \\
C_i A_i^{r} B_i & C_i A_i^{r-1} B_i & \cdots & \mathbf{0}_{dm} \\
\vdots & \vdots & \ddots & \vdots \\
C_i A_i^{r+H-2} B_i & \cdots & \cdots & C_i A_i^{r-1} B_i
\end{bmatrix}, \label{eq:Theta} \\
\Phi_i &:= [ (C_i A_i^r)^\top, \ldots, (C_i A_i^{r+H-1})^\top ]^\top. \label{eq:Phi}
\end{align}
}
\vspace{-.25in}

\noindent Specifically, for a reference barycenter $\bar{Q}_i^{k|r:H}$, the cost $J$ is equivalently expressed as 
\begin{equation}
\begin{split}
&J = (U_i^{k|H})^\top H_i U_i^{k|H} + 2 f_i^\top U_i^{k|H}, \\
&\text{where }
H_i = (\boldsymbol{\Omega}_i \Theta_i)^\top (\boldsymbol{\Omega}_i \Theta_i) + R, \\
&\qquad\quad f_i = (\boldsymbol{\Omega}_i \Theta_i)^\top \boldsymbol{\Omega}_i (\Phi_i \mathbb{E}[x_i^k] - \bar{Q}_i^{k|r:H}), 
\end{split} \label{eq:H_i and f_i}
\end{equation}
with $\boldsymbol{\Omega}_i := \boldsymbol{\Omega}_i^{k|r:H}$. This quadratic reformulation leads to the following optimal control law.
}

\medskip

\begin{theorem}[Strictly Convex Optimal Control Law] \label{theorem:W_min_opt_con}
{
Consider the stochastic discrete-time LTI system \eqref{eq:dyn} for agent $i$. For the cost function $J(U_i^{k|H})$ defined in \eqref{eq:d2oc_optimization_prob}, the unique unconstrained optimal control sequence $U_i^{k|H\star}$ that minimizes the local Wasserstein-based cost is given by
\begin{equation}
U_i^{k|H\star} = -H_i^{-1} f_i, \label{eq:unc_opt_sol}
\end{equation}
where $H_i$ and $f_i$ are as defined in \eqref{eq:H_i and f_i}.
}

If box constraints on the control input
$u_{\min} \leq u_i^{k+\tau} \leq u_{\max}, \quad \tau=0,\ldots,H-1,$
are imposed, then the constrained optimization becomes the strictly convex quadratic program:
\begin{align}
\min_{U_i^{k|H}} \quad & (U_i^{k|H})^\top H_i U_i^{k|H} + 2 f_i^\top U_i^{k|H}\label{eq:QP}
\end{align}
subject to $u_{\min}^{(H)} \leq U_i^{k|H} \leq u_{\max}^{(H)}$,
where $u_{\min}^{(H)} := \mathbf{1}_H \otimes u_{\min}, \, 
u_{\max}^{(H)} := \mathbf{1}_H \otimes u_{\max}$ with $\mathbf{1}_{H}\in\mathbb{R}^{H}$ is the all-ones vector.
The unique optimal solution \((U_i^{k|H})^*\) satisfies the Karush-Kuhn-Tucker (KKT) conditions:
\begin{align*}
&H_i (U_i^{k|H})^* + f_i + \lambda^+ - \lambda^- = 0, \quad \lambda^+, \lambda^- \geq 0,\\
&\lambda^+ \odot ((U_i^{k|H})^* - u_{\max}^{(H)}) = 0, \,
\lambda^- \odot (u_{\min}^{(H)} - (U_i^{k|H})^*) = 0.
\end{align*}
\end{theorem}

\begin{proof}
From Proposition~\ref{proposition:equiv} with \(\mathbb{E}[Y_i^{k|r:H}] = \Theta_i U_i^{k|H} + \Phi_i \mathbb{E}[x_i^k]\), the Wasserstein cost term can be expanded as
\begin{align*}
&\mathbb{E}  \sum_{h=r}^{H+r-1} (\mathcal{W}_i^{k+h})^2 = \operatorname{tr}( \boldsymbol{\Omega}_i^{k|r:H} \Sigma_{i,Y} (\boldsymbol{\Omega}_i^{k|r:H})^\top ) \\
&+ \mathbb{E} \| \boldsymbol{\Omega}_i^{k|r:H} (\Theta_i U_i^{k|H} + \Phi_i x_i^k - \bar{Q}_i^{k|r:H}) \|^2 + \text{const},
\end{align*}
where \(\Sigma_{i,Y} := \mathrm{diag}([C_i \Sigma_{i,w} C_i^\top + \Sigma_{i,v}, \dots, C_i \Sigma_{i,w} C_i^\top + \Sigma_{i,v}])\in\mathbb{R}^{dH\times dH}\) is the block-diagonal output noise covariance over the prediction window.
Since the trace and constant terms are independent of \(U_i^{k|H}\), it is omitted from the optimization. 

Including the weighted input penalty term \(\|U_i^{k|H}\|_R^2\), the cost becomes
$
J(U_i^{k|H}) = (U_i^{k|H})^\top H_i U_i^{k|H} + 2 f_i^\top U_i^{k|H},
$
where \(H_i\) and \(f_i\) are defined in \eqref{eq:H_i and f_i}.

By Assumption~\ref{assump:controllability}, the pair \((A_i, B_i)\) is controllable, which implies that the matrix \(\Theta_i\) has full column rank or sufficient rank properties ensuring that \((\boldsymbol{\Omega}_i^{k|r:H} \Theta_i)^\top (\boldsymbol{\Omega}_i^{k|r:H} \Theta_i)\) is positive semidefinite. Adding \(R \succ 0\) then guarantees that
$
H_i \succ 0.
$

Consequently, the cost function is strictly convex and admits a unique global minimizer obtained from the first-order condition
$
\nabla_{U_i^{k|H}} J = 2 H_i U_i^{k|H} + 2 f_i = 0,
$
which yields the unconstrained solution in \eqref{eq:unc_opt_sol}.

When box constraints are included, the problem becomes a strictly convex quadratic program. The KKT conditions are both necessary and sufficient for optimality and yield the unique constrained solution \((U_i^{k|H})^*\).
\end{proof}

\begin{corollary}(Optimal Solution with Per-Step Ball Input Constraints \(\|u_i^{k+\tau}\| \leq u_{\max}\))
Consider the optimization problem \eqref{eq:d2oc_optimization_prob} over a finite horizon \(H\),
subject to per-step Euclidean norm constraints
$\|u_i^{k+\tau}\| \leq u_{\max}, \quad \forall \tau = 0, \ldots, H-1,
$
where the system evolves according to \eqref{eq:dyn}.

Then, the optimal solution \((U_i^{k|H})^*\) is obtained by projecting each block \(u_i^{k+\tau, \mathrm{uncon}}\) of the unconstrained solution \((U_i^{k|H})^{\mathrm{uncon}}\) onto the Euclidean ball of radius \(u_{\max}\):
\vspace{-0.25in}

{\small
\[
u_i^{k+\tau,*} = \min\left(1, \frac{u_{\max}}{\|u_i^{k+\tau,\mathrm{uncon}}\|}\right) u_i^{k+\tau,\mathrm{uncon}},
\]
}
where \(u_i^{k+\tau, \mathrm{uncon}}\) is the \((k+\tau)\)-th \(m\)-dimensional block of the unconstrained solution defined in \eqref{eq:unc_opt_sol}.
\end{corollary}

This result leverages the separable nature of the per-step Euclidean norm constraints, enabling a closed-form projection of the unconstrained solution without the need for iterative solvers.

The proposed density-driven optimal control law is implemented via a model predictive control (MPC) scheme. At each time step, an optimal control sequence over a finite horizon \(H\) is computed, with only the first input applied before re-optimizing at the next step using updated agent states and local sample information.

\medskip
\begin{remark}
To track the reference distribution, each agent updates its local target samples and corresponding barycenter at every step. However, frequent updates can induce oscillations due to shifting targets. This is mitigated by the receding-horizon structure, which tempers responsiveness, and by an input regularization term that smooths motion while preserving adaptability to local changes.
\end{remark}

Since the effectiveness of D$^2$OC depends critically on the choice of local target samples and their barycenter, the next section presents a detailed analysis based on the mean reachable set.
%%%%%%%%%%%%%%%%%%%%%%%%%%%%%%%%%%%%%%%%%%%%%%%%%%%%%%%%%%%%%
\section{Reachability-Aware Target Sample Selection and Adjustment}
\vspace{-0.116in}

For systems with output relative degree \(r\), control inputs do not affect the output until time \(k + r\). Thus, mean reachability analysis over \(h = r, \ldots, H + r - 1\) is crucial for selecting local target points and computing their barycenter. This section presents a framework for constructing mean reachable output sets, selecting target samples accordingly, and computing a reachable barycenter that reflects future agent capabilities. An adjustment procedure is also provided for cases where the barycenter lies outside the reachable set.

\subsection{\(h\)-Step Reachable Set Analysis}
\vspace{-0.116in}

Given the current state \(x_i^k\) at time \(k\), and assuming an output relative degree \(r\), we consider the \emph{stochastic \(h\)-step reachable set} for \(h \in \{r, \ldots, H + r - 1\}\) at confidence level \(\alpha\), defined in \cite{fiacchini2021probabilistic} as
\begin{equation}
\begin{aligned}
\mathcal{R}_i^{k+h}(\alpha) &:= \{\, x \in \mathbb{R}^n \mid \exists \{u_i^k, \dots, u_i^{k+h-1}\} \subseteq \mathcal{U}, \\
& (x - \mu_i^{k+h})^\top \Sigma_{i,w,h}^{-1} (x - \mu_i^{k+h}) \leq \chi^2_{n,\alpha} \, \},
\end{aligned}
\end{equation}
where \(\mu_i^{k+h} := \mathbb{E}[x_i^{k+h}] = A_i^h \mu_i^k + \sum_{\tau=0}^{h-1} A_i^{h-1-\tau} B_i u_i^{k+\tau}\) is the mean \(h\)-step state prediction under nominal dynamics, \(\Sigma_{i,w,h}\) is the aggregated process noise covariance over \(h\) steps, and \(\chi^2_{n,\alpha}\) denotes the chi-squared quantile at confidence level \(\alpha\) in dimension \(n\). The control set \(\mathcal{U} \subset \mathbb{R}^{mh}\) is assumed convex and compact.

The corresponding \emph{mean reachable state set} at time \(k+h\) is defined as
\vspace{-0.35in}

{\small
\begin{equation}
\mathcal{M}_i^{k+h} := \left\{ \mu_i^{k+h} = A_i^h \mu_i^k + \sum_{\tau=0}^{h-1} A_i^{h-1-\tau} B_i u_i^{k+\tau} \mid u_i^{k+\tau} \in \mathcal{U} \right\}.
\label{eq:reachable output set}
\end{equation}
}
Under Assumption~\ref{assump:controllability}, this set spans a rich subset of the state space consistent with input constraints, enabling meaningful target selection and barycenter computation. Thus, \(\mathcal{R}_i^{k+h}(\alpha)\) can be viewed as a union of confidence ellipsoids centered at points in \(\mathcal{M}_i^{k+h}\), capturing control variability and stochastic disturbances over \(h\) steps.

\subsection{Local Sample Selection and MPC Implementation}
\vspace{-0.116in}

Due to the output relative degree \(r\), the reachable output set remains a singleton for the first \(r-1\) steps, determined solely by the open-loop propagation \(\mathbb{E}[y_i^{k+h}] = C_i A_i^h \mu_i^k\) for all \(h < r\). As a result, the reachability and sample selection become meaningful only from time step \(k + r\) onward.

{To resolve the indeterminacy between the future agent trajectory and the transport plan,} let \(A_i^h \mu_i^k\) denote the {fixed} open-loop nominal prediction of the state at time \(k+h\), \(h=r,\ldots,H+r-1\), obtained by propagating the current state mean \(\mu_i^k\) forward \(h\) steps with zero input. {By treating these nominal predictions as deterministic reference points,} we propose selecting local target points \(\mathcal{Q}_i^{k+h} := \{q_j\}_{j \in \mathcal{S}_i^{k+h}}\) centered around this nominal prediction, so that their barycenter \(\bar{q}_i^{k+h}\) lies close to or within the \(h\)-step mean reachable output set \(C_i\mathcal{M}_i^{k+h}\).

To this end, we adopt the following local optimal transport-based local sample selection strategy:
\begin{align*}
\min_{q_j,\, \pi_j^{k+h}} \sum_{j \in \mathcal{S}_i^{k+h}} \pi_j^{k+h} \| A_i^h \mu_i^k - q_j \|^2
\end{align*}
\vspace{-0.25in}

{subject to \(\sum_{j \in \mathcal{S}_i^{k+h}} \pi_j^{k+h} = \alpha_i^{k+h} := \frac{1}{M_i}\), where \(\alpha_i^{k+h}\) is the constant mass {allocated to} agent \(i\) at each step, and \(\pi_j^{k+h}\) is the transport plan. Since \(A_i^h \mu_i^k\) is fixed, an analytic solution exists that allocates mass to nearest neighbors \(q_j\) sequentially, respecting sample capacities \(\beta_{i,j}^k\), until the transported mass reaches \(\frac{1}{M_i}\) as in \cite{kabir2021wildlife}.
}

According to Proposition~\ref{proposition:equiv}, the Wasserstein cost is minimized when the expected outputs coincide with the barycenter \(\bar{q}_i^{k+h}\) for \(h \in \{r, \ldots, H + r - 1\}\). This requires the barycenter to lie within the mean reachable output set \(C_i \mathcal{M}_i^{k+h}\), which may not always hold. In the following, we thus discuss how to handle cases where \(\bar{q}_i^{k+h}\) lies outside \(C_i \mathcal{M}_i^{k+h}\) to ensure feasibility.

\subsection{Reachable Barycenter Approximation under MPC}
\vspace{-0.116in}
\subsubsection{Projection onto Feasible Output Set}
\vspace{-0.116in}

If the selected barycenter \(\bar{q}_i^{k+h}\) lies outside the mean reachable output set \(C_i \mathcal{M}_i^{k+h}\), we project it onto the closest feasible output:
\begin{equation}
\tilde{q}_i^{k+h} := \operatorname*{argmin}_{z \in C_i\mathcal{M}_i^{k+h}} \| z - \bar{q}_i^{k+h} \|^2.\label{eq:q_proj}
\end{equation}
This projection ensures that there exists a feasible control sequence that reaches \(\tilde{q}_i^{k+h}\), acknowledging MPC's receding horizon limitation.

\subsubsection{Sample Selection with Soft Constraint}
\vspace{-0.116in}
Alternatively, we incorporate feasibility into sample selection via a soft constraint. Denoting \(\bar{q}_i^{k+h}\) as the empirical mean of selected target points \(\{q_j\}\), we solve:
\vspace{-.25in}
{\small
\begin{align*}
\min_{\{u_i^{k+\tau}\}, \{q_j\}} \, & \sum_{j \in \mathcal{S}_i^{k+h}} \pi_j^{k+h}\|q_j - \bar{q}_i^{k+h}\|^2 + \lambda \left\| \bar{q}_i^{k+h} - \hat{q}_i^{k+h} \right\|^2 \\
\text{s.t.} \quad  q_j &\in C_i \mathcal{M}_i^{k+h}, \,\, \bar{q}_i^{k+h} := \frac{ \sum_{j \in \mathcal{S}_i^{k+h}} \pi_j^{k+h} q_j }{ \sum_{j \in \mathcal{S}_i^{k+h}} \pi_j^{k+h} }, \\
\hat{q}_i^{k+h} &:= C_i \Big( A_i^h x_i^k + \sum_{\tau=0}^{h-1} A_i^{h-1-\tau} B_i u_i^{k+\tau} \Big),
\end{align*}
}
\vspace{-.35in}

where \(\lambda > 0\) controls the trade-off between proximity to the desired barycenter and feasibility. {Integrating the reachability-aware target selection and adjustment procedures discussed in this section, the complete three-stage $D^2$OC framework is summarized in Algorithm~\ref{alg:d2oc}.}

\begin{algorithm}[h]
{
\footnotesize
\caption{Decentralized $D^2$OC with Reachability-Aware MPC}
\label{alg:d2oc}
\begin{algorithmic}[1]
\State \textbf{Initialize:} Each agent $i$ sets $\mu_i^0$ and initial coverage weights $\beta_{i,j}^0$.
\For{each time step $k = 0, 1, 2, \dots$}
    \State \textbf{Step 1: Reachability-Aware Target Selection and MPC (Sec. 3 \& 4)}
    \For{$h = r$ to $H+r-1$} \Comment{Receding horizon $H$-step analysis}
        \State Compute nominal prediction $A_i^h \mu_i^k$ and identify local samples $\mathcal{Q}_i^{k+h}$.
        \State Obtain barycenter $\bar{q}_i^{k+h}$ via OT; apply projection \eqref{eq:q_proj} or soft constraint.
    \EndFor
    \State Compute the optimal control $u_i^k$ by solving \eqref{eq:d2oc_optimization_prob}
    \State Apply $u_i^k$ and update state mean $\mu_i^{k+1}$.
    
    \State \textbf{Step 2: Weight Update}
    \State Solve local Wasserstein distance to update sample weights $\beta_{i,j}^{k+1}$.
    \State Reduce weights of covered samples to prioritize under-explored areas.
    
    \State \textbf{Step 3: Weight Sharing and Consensus}
    \State Exchange $\beta_{i,j}^{k+1}$ with neighbors within communication range.
    \State Synchronize via min-weight consensus: $\beta_{i,j} \leftarrow \min(\beta_{i,j}, \beta_{neighbor,j})$.
\EndFor
\end{algorithmic}
}
\end{algorithm}

%%%%%%%%%%%%%%%%%%%%%%%%%%%%%%%%%%%%%%%%%%%%%%%%%%%%%%%%%%

\section{Convergence Analysis}\label{sec:convergence}
\vspace{-0.116in}

Now the convergence behavior of the proposed control strategy is analyzed under this section. The main result shows that the global output distribution converges toward the target in the mean squared Wasserstein sense with a certain bound.

\medskip
\begin{theorem}
\label{thm:convergence_upper}
(Convergence of D$^2$OC with MPC and Decentralized Barycenter Updates)

Consider a multi-agent system with discrete-time stochastic linear dynamics \eqref{eq:dyn} and decentralized, local communication. Let the output relative degree be \(r\), and define the reachable output barycenter set at time \(k+h\), for \(h = r, \ldots, H + r - 1\), as \(C_i \mathcal{M}_i^{k+h}\), where \(\mathcal{M}_i^{k+h}\) is given by \eqref{eq:reachable output set}, \(\mu_i^k = \mathbb{E}[x_i^k]\), and \(\mathcal{U}\) is a compact input set.

At each \(k\), an MPC solves the finite-horizon problem \eqref{eq:d2oc_optimization_prob} over horizon \(H\) to minimize deviation between predicted outputs and target barycenters \(\{\bar{q}_i^{k+h}\}_{h=r}^{H+r-1}\), applying only the first input \(u_i^k\) before re-optimization.

Each agent determines these barycenters from local sample points with weights \(\{\beta_{i,j}^k\}\), updated by (i) local coverage and (ii) decentralized communication through pointwise minimum of weights between agents within communication range.

Define the empirical multi-agent output distribution \(\rho^k\) and let \(\nu\) be the target distribution as shown in \eqref{eq:agent_distribution}. Then, under decentralized D$^2$OC with local weight sharing, 
\begin{align}
\mathbb{E}[\mathcal{W}_2^2(\rho^{k+1}, \nu)] \leq \left(1 - \frac{c}{k+1}\right) \mathbb{E}[\mathcal{W}_2^2(\rho^k, \nu)] + \frac{C}{(k+1)^2},\label{eq:convergence}
\end{align}
for constants \(c, \, C > 0\), where \(c\) depends on controllability and convergence rate, and \(C\) bounds projection and noise effects.

Consequently,
$
\lim_{k \to \infty} \mathbb{E}[\mathcal{W}_2^2(\rho^k, \nu)] \leq \epsilon_h,
$
for some \(\epsilon_h > 0\) determined by communication frequency, noise, and projection errors.
\end{theorem}
\vspace{-.15in}

\begin{proof}
At each step \(k\), the MPC solves the finite-horizon problem \eqref{eq:d2oc_optimization_prob}, which is strictly convex under box input constraints (Theorem \ref{theorem:W_min_opt_con}), applying only the first input before re-optimizing. Deviations arise from: (i) stochastic noise, and (ii) dynamic local weight updates.

Let the mean reachable output set at \(k+h\) be \(C_i \mathcal{M}_i^{k+h}\), and \(\bar{q}_i^{k+h}\) be the target barycenter selected from target samples.

Under decentralized communication, agents update sample weights \(\beta_{i,j}^k\) by taking the minimum of local and received weights within communication range. This conservative update aligns barycenters across neighbors, reducing redundant coverage and enhancing distributional alignment with \(\nu\).

Then, the error is defined for \(h = r, \ldots, H + r - 1\) by
$
e_i^{k+h} := y_i^{k+h} - \bar{q}_i^{k+h} = C_i x_i^{k+h} - \bar{q}_i^{k+h}.
$
By triangle inequality, we have
\[
\|e_i^{k+h}\| \leq \|C_i x_i^{k+h} - \tilde{q}_i^{k+h}\| + \|\tilde{q}_i^{k+h} - \bar{q}_i^{k+h}\|,
\]
where \(\tilde{q}_i^{k+h}\) is the projection of \(\bar{q}_i^{k+h}\) onto the mean reachable output set \(C_i \mathcal{M}_i^{k+h}\).

(i) Projection error: If \(\bar{q}_i^{k+h} \in C_i \mathcal{M}_i^{k+h}\), then \(\|\tilde{q}_i^{k+h} - \bar{q}_i^{k+h}\|=0\); otherwise, it is bounded by some \(\epsilon\).
\vspace{-.05in}

(ii) Tracking error: Writing \(x_i^{k+h} = \hat{x}_i^{k+h} + \xi_i^{k+h}\) with nominal \(\hat{x}_i^{k+h}\) and noise \(\xi_i^{k+h}\), we have
\vspace{-.05in}
\[
\|C_i x_i^{k+h} - \tilde{q}_i^{k+h}\| \leq \|C_i \hat{x}_i^{k+h} - \tilde{q}_i^{k+h}\| + \|C_i \xi_i^{k+h}\|.
\]
\vspace{-.05in}
By construction, \(\tilde{q}_i^{k+h} \in C_i \mathcal{M}_i^{k+h}\), so the nominal output \(C_i \hat{x}_i^{k+h}\) can be steered arbitrarily close to \(\tilde{q}_i^{k+h}\) under the given dynamics. For the stochastic part, \(\mathbb{E}[\|C_i \xi_i^{k+h}\|^2] \leq \lambda_{\max}(C_i \Sigma_\xi C_i^\top)\), where \(\Sigma_\xi = \sum_{\ell=0}^{h-1} A_i^{h-1-\ell} \Sigma_w (A_i^{h-1-\ell})^\top\), and \(\lambda_{\max}(\cdot)\) denotes the largest eigenvalue of a symmetric positive semi-definite matrix. With finite \(H\), \(\Sigma_\xi\) is bounded under Assumption~\ref{assump:stability}, ensuring uniform boundedness of \(\mathbb{E}[\|e_i^{k+h}\|]\) with a bound depending on \(H\).
\vspace{-.05in}

{
(iii) Wasserstein descent: The bounded tracking error established in (i) and (ii) ensures a dissipative drift. Under limited communication, the decentralized min-weight consensus mitigates overestimation of coverage weights. By reducing variance in $\rho^k$, this mechanism promotes consistent alignment with the target distribution $\nu$.

To derive the recurrence relation in \eqref{eq:convergence}, we utilize the update law $\rho^{k+1} = \frac{k}{k+1}\rho^k + \frac{1}{k+1}\delta_{y_i^{k+1}}$. By leveraging the $L$-smoothness of the squared Wasserstein distance in the sense of Fr\'{e}chet derivatives \cite{gelbrich1990formula,gupta2021convergence}, $\mathbb{E}[\mathcal{W}_2^2(\rho^{k+1}, \nu)]$ is expanded as:
\begin{align}
    \mathcal{W}_2^2(\rho^{k+1}, \nu) &\leq \mathcal{W}_2^2(\rho^k, \nu) + \tfrac{2}{k+1} \langle \nabla \mathcal{W}_2^2(\rho^k, \nu), \delta_{y_i^{k+1}} - \rho^k \rangle \nonumber \\
    &\quad + \tfrac{L}{(k+1)^2} \|\delta_{y_i^{k+1}} - \rho^k\|^2 \label{eq:expansion_compressed}
\end{align}

By taking the total expectation on both sides of \eqref{eq:expansion_compressed} and applying the dissipative drift property 
$\mathbb{E}[\langle \nabla \mathcal{W}_2^2, \delta_{y_i^{k+1}} - \rho^k \rangle \mid \mathcal{F}_k] \leq -\tilde{c} \mathcal{W}_2^2(\rho^k, \nu) + \epsilon_{proj}$, 
the recurrence relation in \eqref{eq:convergence} is directly obtained. Here, $\mathcal{F}_k$ is the filtration up to time $k$, $\tilde{c} > 0$ is the contraction gain, and $\epsilon_{proj} \geq 0$ denotes the residual error bound from control projection and discretization.
By defining $c := 2\tilde{c} > 1$ and a constant $C > 0$ that bounds the second-order expansion and projection terms, applying Chung's Lemma \cite{chung1954stochastic} directly yields an $O(1/k)$ convergence rate, specifically $\mathbb{E}[\mathcal{W}_2^2(\rho^k, \nu)] \leq \frac{C}{c-1}k^{-1} + o(k^{-1})$. This ensures that
$
    \lim_{k \to \infty} \mathbb{E}[\mathcal{W}_2^2(\rho^k, \nu)] \leq \epsilon_h,
$
where the residual error $\epsilon_h \geq 0$ is determined by the communication frequency, the intensity of process/measurement noises, and projection errors associated with physical constraints.
}
\end{proof}
\vspace{-.15in}

This convergence in expected squared Wasserstein distance implies that the \textit{mean squared error} between the empirical output distribution and the reference distribution vanishes asymptotically up to \(\epsilon_h\).
\vspace{-.05in}

%%%%%%%%%%%%%%%%%%%%%%%%%%%%%%%%%%%%%%%%%%%%%%%%%%%%%%%
\section{Simulation Results}
\vspace{-0.116in}

To validate the proposed method, simulation results are presented in this section. Three linearized quadrotors are considered for the multi-agent platform. The state of agent \(i\) at time step \(k\) is defined by
\(x_i^k := [\mathrm{x}_i^k, \dot{\mathrm{x}}_i^k, \mathrm{y}_i^k, \dot{\mathrm{y}}_i^k, \mathrm{z}_i^k, \dot{\mathrm{z}}_i^k, \phi_i^{k}, \dot{\phi}_i^{k}, \theta_i^{k}, \dot{\theta}_i^{k}, \psi_i^k, \dot{\psi}_i^k]^\top,\) 
where \(\mathrm{x}_i^k\), \(\mathrm{y}_i^k\), and \(\mathrm{z}_i^k\) denote position coordinates, and \(\phi_i^k\), \(\theta_i^k\), and \(\psi_i^k\) represent roll, pitch, and yaw angles, respectively. 
Dotted symbols indicate the corresponding translational or angular velocities in discrete time. 
The four control inputs for agent \(i\) are \(u_{i,1}^k = \tau_{i,\phi}^k\), \(u_{i,2}^k = \tau_{i,\theta}^k\), \(u_{i,3}^k = \tau_{i,\psi}^k\), and \(u_{i,4}^k = T_i^k\), 
corresponding to torques for roll, pitch, yaw, and total thrust. 
Each input is box-constrained to reflect motor torque and thrust limits.
{and the communication range for decentralized weight-sharing is set to 5.}
\vspace{-0.1in}

{The discretized quadrotor model exhibits an output relative degree of $r = 4$, satisfying $C_i A_i^{\ell-1} B_i = 0$ for $\ell = 1, 2, 3$. This implies that control inputs at time $k$ begin to affect the system output only at time $k+4$. To evaluate the robustness of the proposed D$^2$OC under severe uncertainty, substantial stochasticity is incorporated with process noise $w_i^k \sim \mathcal{N}(0, 0.2 \times \mathbf{I}_{12})$ and measurement noise $v_i^k \sim \mathcal{N}(0, 0.5 \times \mathbf{I}_3)$. Furthermore, the initial states are randomized as $x_i^0 = \bar{x}_i^0 + \delta_i$ with $\delta_i \sim \mathcal{N}(0, 4 \times \mathbf{I}_{12})$ to represent initial spatial uncertainty in every simulation run. With the prediction horizon $H = 1$, the optimal control input is obtained using \texttt{quadprog} in \textsc{MATLAB} by solving the quadratic program defined in \eqref{eq:H_i and f_i}. When the target barycenter is outside the reachable set, it is projected to the closest feasible output as in \eqref{eq:q_proj}.}

\begin{figure}[!h]
    \centering
    \subfloat[Conventional D$^2$C]{
    \includegraphics[trim=70 0 70 20,width=0.475\linewidth]{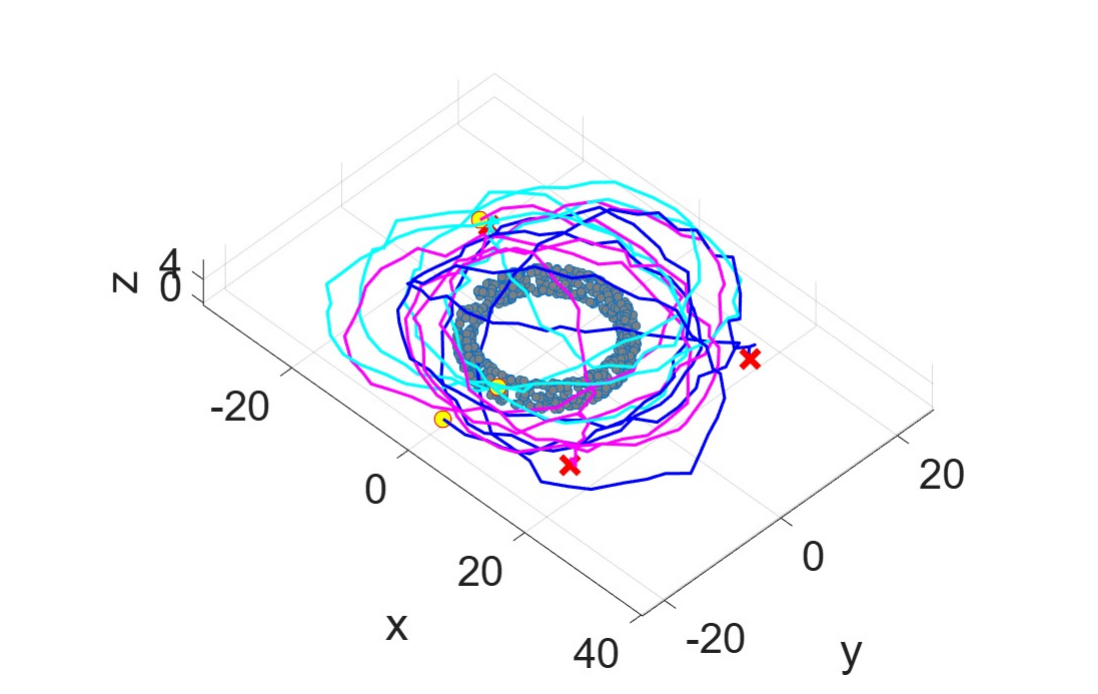}}\;
    \subfloat[Proposed D$^2$OC]{
    \includegraphics[trim=70 0 70 20,width=0.475\linewidth]{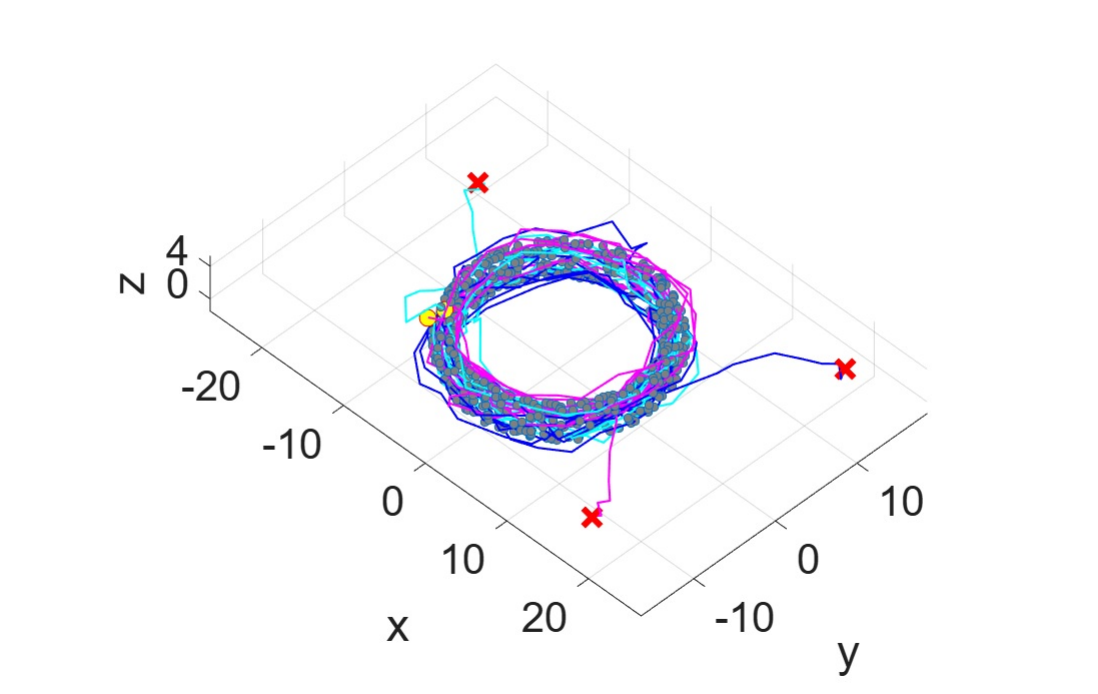}}
    \caption{{Comparison of 3D quadrotor trajectories under (a) Conventional D$^2$C and (b) Proposed Stochastic D$^2$OC. Solid lines represent agent paths, while `$\times$' and `$\circ$' markers denote initial and final positions, respectively.}}
    \label{fig:sim trajectory}
\end{figure}

{The comparative performance between the conventional D$^2$C approach \cite{kabir2021wildlife,lee2022density} and the proposed method is illustrated in Fig.~\ref{fig:sim trajectory}. In the conventional approach (Fig.~\ref{fig:sim trajectory}(a)), a goal point is greedily selected from sample points closest to the agent and tracked via a standard PID-type controller. This method fails to accurately reconstruct the torus-shaped distribution (gray dots), with trajectories drifting toward the outer boundaries due to the high sensitivity of the greedy logic to injected noise. In contrast, the proposed Stochastic D$^2$OC in Fig.~\ref{fig:sim trajectory}(b) demonstrates superior performance, precisely matching the reference distribution despite severe noise and initial state uncertainties.}

\begin{figure}[!h]
    \centering
    \subfloat[Communication Events]{
    \includegraphics[width=0.45\linewidth]{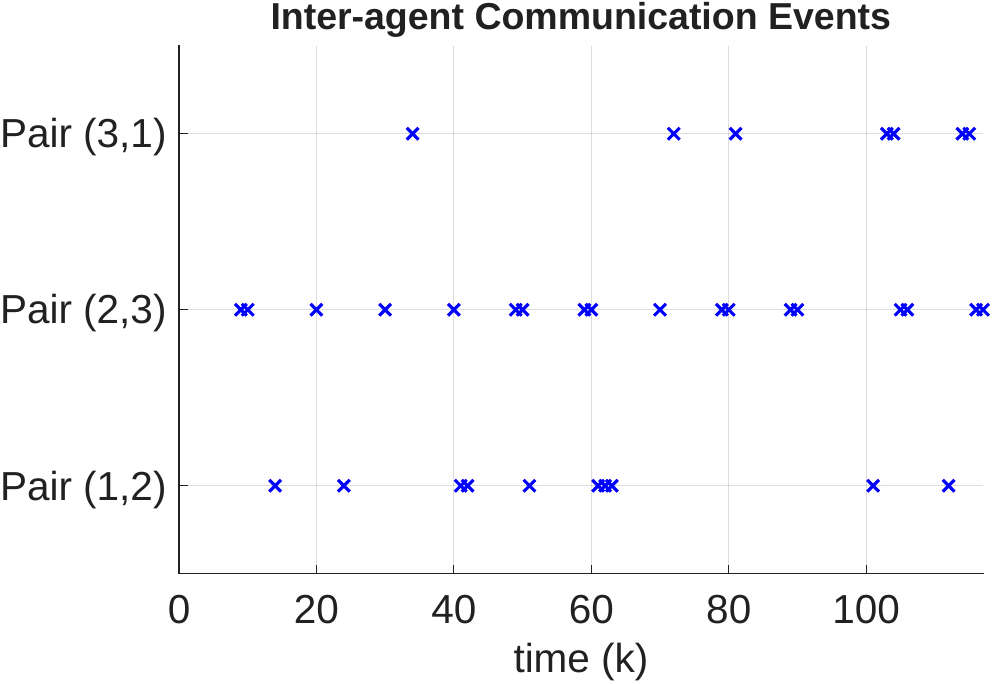}}\;\;
    \subfloat[Expected $W_2^2$ distance]{
    \includegraphics[width=0.45\linewidth]{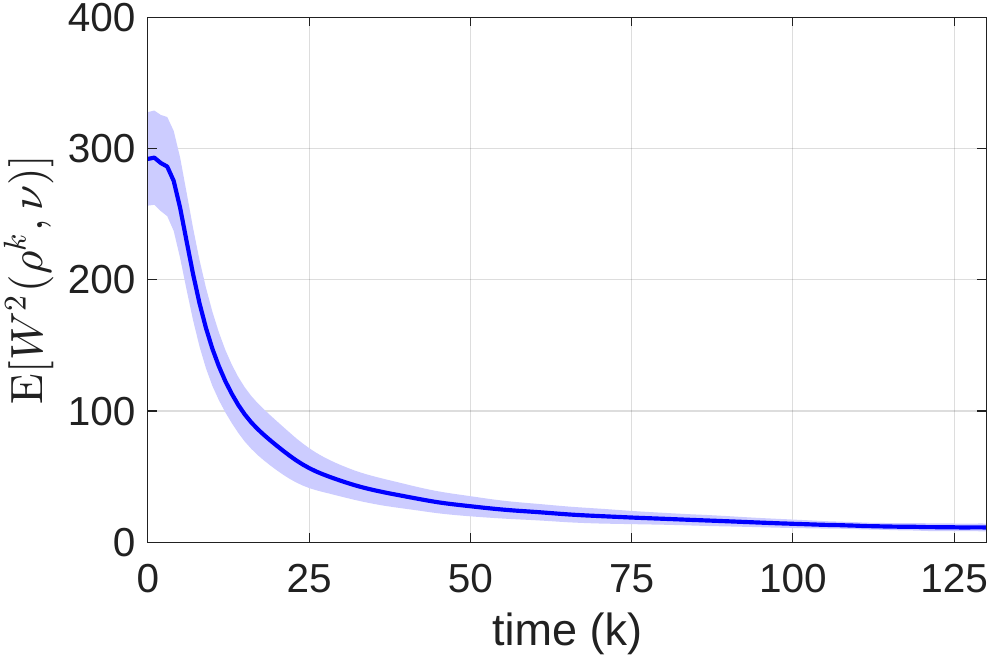}}
    \caption{{Performance analysis of the proposed Stochastic D$^2$OC. (a) Communication events between agents. (b) Empirical mean (solid line) and standard deviation (shaded region) of $W_2^2$ over 100 simulation runs.}}
    \label{fig:sim result}
\end{figure}

{To further evaluate the performance in decentralized and stochastic environments, Fig.~\ref{fig:sim result} provides a detailed analysis. Fig.~\ref{fig:sim result}(a) shows intermittent communication events, confirming that D$^2$OC achieves time-averaged distribution matching even with occasional weight-sharing. The statistical results in Fig.~\ref{fig:sim result}(b) show that the empirical mean (solid blue line) asymptotically converges toward a bounded region $\epsilon_h$, consistent with Theorem~\ref{thm:convergence_upper}. Notably, the narrow standard deviation (shaded region) demonstrates the high robustness of the proposed method, ensuring consistent performance despite the presence of large process/measurement noises and injected initial state perturbations across 100 simulation runs.}

%%%%%%%%%%%%%%%%%%%%%%%%%%%%%%%%%%%%%%%%%%%%%%%%%%%%%%%
\section{Conclusion}
\vspace{-0.116in}

This paper introduced Density-Driven Optimal Control, a novel framework for solving the multi-agent non-uniform area coverage problem under realistic operational constraints. Unlike traditional uniform coverage strategies that are inefficient or infeasible in large-scale or resource-limited settings, D$^2$OC leverages optimal transport theory to derive control laws that steer agents toward matching a task-specific reference distribution. The framework accounts for stochastic LTI system dynamics, agent constraints, and limited communication, providing a unified and principled solution. Theoretical guarantees on convergence were established via reachability analysis, and simulation results demonstrated the effectiveness and practicality of the proposed approach. 
\vspace{-0.075in}

\small
\bibliographystyle{unsrt}
\bibliography{references}

\end{document}